\documentclass{amsart}

\usepackage{graphicx} 
\usepackage{subfigure}
\usepackage{amsmath}
\usepackage{amsthm}
\usepackage{amssymb}
\usepackage{natbib}
\usepackage{comment}
\usepackage{amsfonts}
\usepackage{titlesec}
\titlelabel{\thetitle.\quad}
\usepackage{caption}
\usepackage{float}
\usepackage{color,xcolor,mathrsfs,mathtools,calc,bm,array,graphicx}
\usepackage{multimedia,media9,xmpmulti,siunitx,animate,tikz}

\renewcommand{\Re}{\operatorname{Re}}
\renewcommand{\Im}{\operatorname{Im}}

\newtheorem{theorem}{Theorem}
\newtheorem{lemma}{Lemma}
\newtheorem{proposition}{Proposition}

\theoremstyle{definition}

\newtheorem{example}{Example}

\theoremstyle{remark}

\title[Zeros of Harmonic Functions]{Using Real-Variable Techniques to Study Zeros of Complex-Valued Harmonic Functions}

\author[Brooks]{Jennifer Brooks}
\address{Brigham Young University Mathematics Department, Provo UT 84602 USA}
\email{jbrooks@mathematics.byu.edu}

\author[Jenkins]{Mary Jenkins}
\address{Brigham Young University Mathematics Department, Provo UT 84602 USA}
\email{mjenkins@mathematics.byu.edu}

\author[Liechty]{Samuel Liechty}
\address{Brigham Young University Mathematics Department, Provo UT 84602 USA}
\email{s1iechty@student.byu.edu}

\author[Parker]{Kaden Parker}
\address{Brigham Young University Mathematics Department, Provo UT 84602 USA}
\email{kaden.parker220@gmail.com}

\author[Perez]{Katherine Perez}
\address{Brigham Young University Mathematics Department, Provo UT 84602 USA}
\email{katperez729@gmail.com}

\author[Robison]{Dallin Robison}
\address{Brigham Young University Mathematics Department, Provo UT 84602 USA}
\email{dallinrobison@gmail.com}

\author[Sampson]{Eli Sampson}
\address{Brigham Young University Mathematics Department, Provo UT 84602 USA}
\email{elisampsonj@gmail.com}

\author[Wilson]{Matthew G. Wilson}
\address{Brigham Young University Mathematics Department, Provo UT 84602 USA}
\email{wilson.matthewg@gmail.com}

\keywords{complex-valued harmonic functions, zeros}

\subjclass{30C15}

\begin{document}

\maketitle

\begin{abstract}
    We investigate the zeros of two one-parameter families of harmonic functions and describe how the number of zeros depends on the parameter. Our functions have the property that all zeros lie on certain rays in the complex plane and thus we are able to use real-variable techniques to count the zeros on each ray. 
\end{abstract}

\section{Introduction}

We consider the zeros of two closely-related families of complex-valued harmonic functions. Such functions have the form $f = u + iv$ where both $u$ and $v$ are real-valued harmonic functions. Additionally, they can be expressed as $f = h + \overline{g}$, where $h$ and $g$ are analytic.  Much of the research to date on the zeros of such functions focuses on the polynomial case. The Fundamental Theorem of Algebra states that for non-constant {\em analytic} polynomials, the number of zeros equals the degree. However, for harmonic polynomials, this theorem does not apply and the number of zeros may be much larger than the degree. Many researchers have sought bounds on the number of zeros of these polynomials. Sheil-Small conjectured \cite{sheil} and Wilmhurst \cite{wilmshurst1998valence} proved that for a harmonic polynomial $f = h + \overline{g}$ for which $\deg h = n$, $\deg g= m$, and $m < n$, the maximum number of zeros of $f$ is $n^2$. S\`{e}te and Zur \cite{sete2024zeros} showed that for every $k  = n,n + 1,...,n^2$ there exists a harmonic polynomial of degree $n$ with $k$ zeros.

For {\em analytic} polynomials, the establishment of the Fundamental Theorem of Algebra did not end the study of their zeros; researchers have sought and continue to seek results relating the locations of these zeros to the coefficients. (See \cite{melman1} and the references therein.) Analogously, the establishment of sharp bounds on the number of zeros of complex-valued harmonic polynomials does not end their study. Because the number of zeros itself depends not only on the degree but also on the coefficients, ongoing research addresses both how the number and locations of the zeros vary with the coefficients. (See \cite{sam}, \cite{work}, \cite{legesse2022location}, \cite{melman24}.) Much of this research focuses on simple one- or two-parameter families of functions with the goal of obtaining detailed zero-counting theorems.   

Specifically, Brilleslyper et al.\ \cite{BBDHS} investigated a family of trinomials:
\begin{equation}\label{Brilleslyper}
    q(z) =  z^n + c\overline{z}^k -1
\end{equation}
where $1 \leq k \leq n-1$, $n\geq3,$ $c \in \mathbb{R}_{+}$ and $\gcd(n,k) = 1$. 
Their theorem states that as $c$ increases over the positive reals, the number of zeros increases monotonically from a minimum of $n$ to a maximum of $n + 2k$. Their proof used the Argument Principle for harmonic functions; for this family, the {\em critical curve} separating the sense-preserving and sense-reversing regions is a circle whose image under the function is a hypocycloid. They determined how the winding number about the origin of this hypocycloid varies with $c$ and were therefore able to count the zeros inside the sense-reversing region for all values of $c$. 

We investigate two families of complex-valued harmonic functions. The first is a family of harmonic polynomials of a form similar to \eqref{Brilleslyper}: 
\begin{equation}\label{ourform}
   p(z) = z^m + c(z^k +\overline{z}^k) -1, \qquad \text{where } m > k, \quad c > 0
\end{equation}
A natural question is: How does this addition of a single term affect the number of zeros as $c$ increases? The strategy for counting the zeros of the two families turns out to be quite different.  Unfortunately, because of the additional term in our family of polynomials, the critical curve is non-circular. This renders the strategy used in Brilleslyper et al. \cite{BBDHS} not viable. However, we specifically add the term $c{z}^k$, which is the conjugate of the middle term of \eqref{Brilleslyper}. Because of this addition, $z^m$ is the only non-real term in \eqref{ourform}. Conveniently, we can then reduce much of the problem to counting the zeros of a real-valued polynomial and can apply real variable techniques.

In this paper we prove the following theorem:
\begin{theorem}\label{theorem1}
Let $p(z) = z^m+c(z^k+\overline{z}^k)-1$ where $m > k$, $c \in \mathbb{R_+}$, and $\gcd(m,k) = 1$. As $c$ increases, the number of zeros of $p$ increases monotonically from a minimum of $m$ to a maximum of $m + N$, where N is determined as follows:
    \begin{itemize}
    \item If $m \equiv 0 \mod 4$, $N = m$.
    \item If $m \equiv 1 \mod 4$ and $k$ is odd, $N = m-1$.
    \item If $m \equiv 1 \mod 4$ and $k$ is even, $N = m+1$.
    \item If $m \equiv 2 \mod 4$, $N = m-2$.
    \item If $m \equiv 3 \mod 4$ and $k$ is odd, $N = m+1$.
    \item If $m \equiv 3 \mod 4$ and $k$ is even, $N = m-1$.
\end{itemize}
\end{theorem}
Before we continue, we illustrate with an example: 

\begin{example}
Consider the polynomial $p(z) = z^5 + c(z^4 +\overline{z}^4) -1$ corresponding to $m=5$ and $k=4$. Because $5 \equiv 1 \mod 4$ and $4$ is even, by Theorem \ref{theorem1}, the number of zeros grows from $5$ to $5 + (5+1) = 11$. As $c$ changes from $c = 0.1$ to $c = 1$ and eventually to $c = 3$, the number of zeros increases from $5$ to $7$, then to $11$.  See Figure \ref{fig1}.
\end{example}

\begin{figure}[h]
  \subfigure[$c=0.1$]{
    \includegraphics[scale = .5]{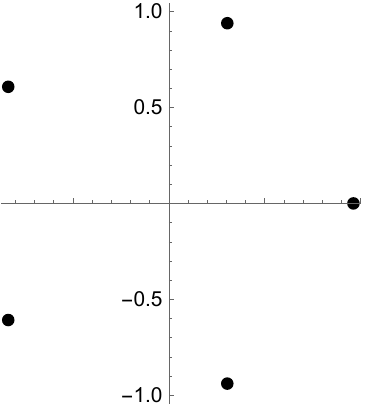}
  }
  \hfill  
  \subfigure[$c=1$]{
    \includegraphics[scale = .65]{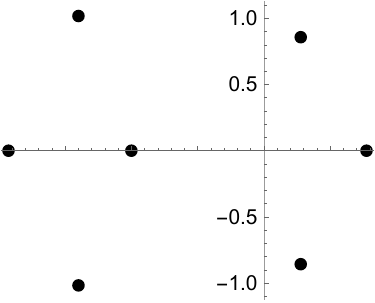}
  }
  \hfill  
  \subfigure[$c=3$]{
    \includegraphics[scale = .55]{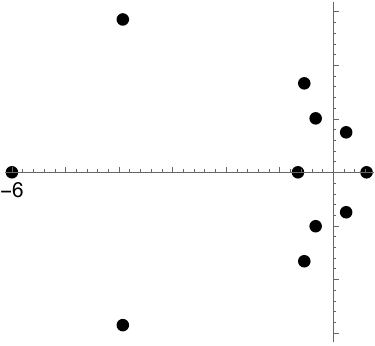}
  }
\caption{\label{fig1} Zeros of $p(z) = z^5 + c(z^4 +\overline{z}^4) - 1$}
\end{figure}

Inspired by the work of Brilleslyper et al., Brooks and Lee \cite{lee} investigated a related family of harmonic trinomials, each having a pole at the origin:
\begin{equation}\label{BrooksnLee}
    f_c(z) =  z^n + \frac{c}{\overline{z}^k} -1
\end{equation}
where $n, k \in \mathbb{N}$ with $n > k$ and $\gcd(n,k) = 1$.  They showed that for this family, the number of zeros decreases from $n+k$ to $n-k$ as $c$ increases through the positive reals. Not surprisingly, their proof strategy was the same as that in Brilleslyper et al. The critical curve separating the sense-preserving and sense-reversing regions is again a circle, but this time, its image under the function is an epicycloid. They determined the winding number of this epicycloid about the origin for all values of $c$ and were therefore able to count the zeros in the sense-reversing region for all values of $c$. 

Just as we added a single term to go from \eqref{Brilleslyper} to \eqref{ourform}, we also investigate a family of harmonic functions related to the family \eqref{BrooksnLee}. However, we write our functions in a way that makes their connection with \eqref{ourform} clearer: 
\begin{equation}\label{SecondForm}
    p(z) = z^m + c\left(z^k + \overline{z}^k \right) -1
\end{equation}
where $m\in \mathbb{N}$, $k \in \mathbb{Z}$ with $k<0$, $m>|k|$, $c \in \mathbb{R_+}$, and $\gcd(m,k)=1$. Because of the additional analytic term, we again have a family of harmonic functions whose critical curve is not a circle. Therefore, once again, a proof using the Harmonic Argument Principle is not viable. However, because the middle two terms are conjugates, once again  $z^m$ is the only non-real term. We can therefore again reduce the complex problem to a real-variable problem. 

In this paper, we prove the following theorem:
\begin{theorem}\label{theorem2}
    Let $p(z) = z^m + c(z^{k} +\overline{z}^{k}) -1$, where $m\in \mathbb{N}$, $k \in \mathbb{Z}$ with $k<0$, $m>|k|$, $c \in \mathbb{R}_+$, and $\gcd(m,k)=1$. As $c$ increases, the number of zeros of $p$ decreases monotonically from a maximum of $M+N$ to a minimum of $M$, where $M$ and $N$ are determined as follows:
    \begin{itemize}
        \item If $m \equiv 0 \mod 4$, $M = m+1$ and $N = m-2$.
        \item If $m \equiv 1 \mod 4$ and $k$ is odd, $M = m-1$ and $N = m+1$.
        \item If $m \equiv 1 \mod{4}$ and k is even, $M = m, N = m + 1$
        \item If $m \equiv 2 \mod 4$, $M=m-1$ and $N=m$.
        \item If $m \equiv 3 \mod 4$ and $k$ is odd, $M=m+1$ and $N=m-1$.
        \item If $m \equiv 3 \mod{4}$ and $k$ is even, $M=m$ and $N=m-1$. 
    \end{itemize}

\end{theorem}

We illustrate with an example:
\begin{example}
Consider the function $p(z) = z^{5} + c(z^{-4} +\overline{z}^{-4}) -1$ corresponding to $m=5$ and $k=-4$. Because $5 \equiv 1 \mod 4$ and $-4$ is even, by Theorem \ref{theorem2}, $M = 5$ and $N = 6$, and so the number of zeros decreases from $5+6 = 11$ to $5$. When $c$ changes from $c = 0.1$ to $c = 0.2$ and eventually to $c = 1$, the number of zeros decreases from $11$ to $9$, then to $5$. See Figure \ref{fig2}. 
\end{example}

\begin{figure}[H]
  \subfigure[$c=0.1$]{
    \includegraphics[scale = 0.6]{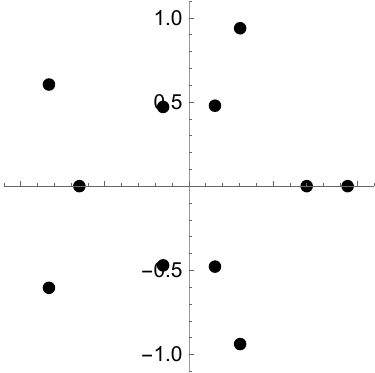}
  }
  \hfill  
  \subfigure[$c=0.2$]{
    \includegraphics[scale = 0.6]{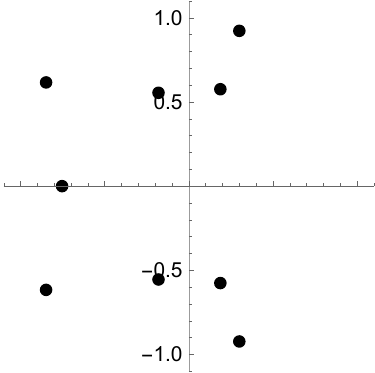}
  }
  \hfill  
  \subfigure[$c=1$]{
    \includegraphics[scale = 0.6]{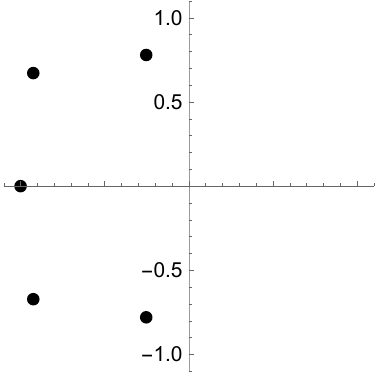}
  }
\caption{ \label{fig2} Zeros of $p(z)=z^5 + c(z^{-4} +\overline{z}^{-4}) - 1$ }
\end{figure}

Our paper is structured as follows: In Section 2 we show that the zeros of \eqref{ourform} and \eqref{SecondForm} all lie on one of $2m$ rays in the complex plane, and we show how to reduce the question of counting zeros on each ray to a question of counting positive real zeros of a certain real-valued function associated with the ray. These real-valued functions take one of six forms, and in Section 3, we prove lemmas that count the positive real zeros for each.  Lastly, in Section 4 we count the rays giving rise to each of these cases in order to prove Theorems \ref{theorem1} and \ref{theorem2}.

\section{Location of Zeros}

The two families \eqref{ourform} and \eqref{SecondForm} can be combined into a single family
\begin{equation}\label{combinedfam}
    p(z) = z^m+c (z^k +\overline{z}^k)-1, \quad m \in \mathbb{N}, \; k \in \mathbb{Z} \setminus\{0\}
\end{equation}
where $c \in \mathbb{R}_+$, $m>|k|$, and $\gcd(m,k)=1$. For some parts of our setup, the sign of $k$ does not matter whereas for other parts it does.  

As our first step, for any non-zero integer $k$, we write $z$ in polar form and separate $p$ into real and imaginary parts.
\begin{align*}
p(z) &= z^m+c(z^k+\overline{z}^k)-1 \\
&=  (re^{i\theta})^m+c\left((re^{i\theta})^k+(re^{-i\theta})^k\right)-1 \\
&= r^m(\cos{m\theta}+i\sin{m\theta})+2cr^k\cos{k\theta}-1. 
\end{align*}
Thus, 
\[
\Re p(re^{i\theta})  = r^m\cos{m\theta} + 2cr^k\cos k\theta - 1
\]
and
\[
\Im p(re^{i\theta})  = r^m\sin m\theta .
\]


Because a complex-valued function is 0 if and only if both the real and imaginary parts are 0 and because $z=0$ is not a zero of $p$, the next lemma follows immediately 

\begin{lemma}\label{lemma1}
Every zero of $p(z) = z^m+c(z^k+\overline{z}^k)-1$, $m \in \mathbb{N}$, $k \in \mathbb{Z} \setminus\{0\}$ lies on one of the $2m$ rays in the complex plane with angle $\theta=\frac{j\pi}{m} \text{ for } j=0,1,2,...,2m-1$. 
\end{lemma}

We illustrate this lemma by showing again the plots of the zeros of the function from Example 1, this time with the rays included. 

\begin{figure}[H]
  \subfigure[$c=0.1$]{
    \includegraphics[scale = 0.5]{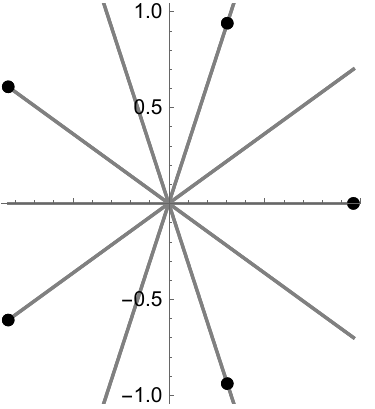}
  }
  \hfill  
  \subfigure[$c=1$]{
    \includegraphics[scale = 0.65]{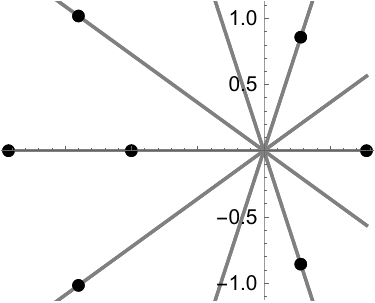}
  }
  \hfill  
  \subfigure[$c=3$]{
    \includegraphics[scale = 0.55]{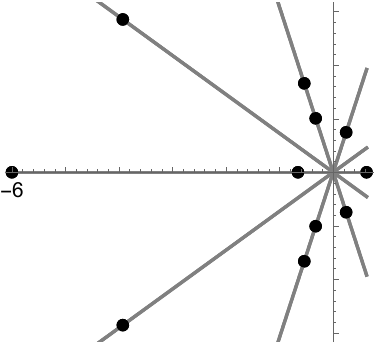}
  }
\caption{\label{fig3} Zeros and rays of $p(z) = z^5 + c(z^4 +\overline{z}^4) - 1$}
\end{figure}



By Lemma \ref{lemma1}, we see that if $z = re^{i\theta}$ is a zero of $p$, then 
\begin{align*}
\Re p(z) 
 &= r^m\cos\left({\frac{mj\pi}{m}}\right) + 2cr^k\cos\left({\frac{kj\pi}{m}}\right) - 1 \\
 &=  (-1)^j r^m + 2c\cos\left({\frac{kj\pi}{m}}\right) r^k - 1\\
 &:= f_j(r).
\end{align*}
To simplify notation, let $\alpha = \cos({\frac{kj\pi}{m}})$. 

Because $r$ is the modulus of a non-zero complex number, we see that counting the zeros of $p$ on the $j$-th ray is equivalent to counting the positive real zeros of $f_j$. The number of zeros of $f_j$ in turn depends on the sign of $k$, the parity of $j$, and whether $\alpha$ is positive, negative, or zero.  In the next section, we prove lemmas on the number of positive real zeros of $f_j$ in each case. Later, we count the $j \in \{0,1,2,...,2m-1\}$ giving rise to each case. We then use these two pieces of information to prove the theorems. 

\section{Positive real zeros of $f_j$}

For all of the lemmas to follow, 
\begin{equation}\label{realfunctions}
    f(r) = (-1)^j r^m +2 c \alpha r^k - 1
\end{equation}
for $m \in \mathbb{N}$, $k \in \mathbb{Z} \setminus \{0\}$, $m>|k|$, $c \in \mathbb{R}_+$, $j \in \{0,\ldots, 2m-1\}$, and $|\alpha| \leq 1$. 

The cases in which $\alpha = 0$ are easy and independent of $k$. We summarize the results here for completeness. 

\begin{lemma}\label{degposlemma}
Let $f$ be as in \eqref{realfunctions} with $j$ even and $\alpha =0$. Then for all $c>0$, $f$ has exactly one positive real zero. 
\end{lemma}

\begin{proof}
Here, $f(r) = r^m - 1$. Clearly the only positive real zero is at $r=1$. 
\end{proof}

\begin{lemma}\label{degneglemma}
Let $f$ be as in \eqref{realfunctions} with $j$ odd and $\alpha = 0$. Then for all $c>0$, $f$ has no positive real zeros. 
\end{lemma}

\begin{proof}
Here $f(r) = -r^m -1$, which clearly has no postive real zeros. 
\end{proof}

For the other four cases, the sign of $k$ matters, and so we consider these lemmas in two subsections. 

\subsection{Positive $k$}

\begin{lemma}\label{kadenlemma}
Let $f$ be as in \eqref{realfunctions} with $k>0$, $j$ even, and $\alpha >0$. Then for all $c>0$, $f$ has exactly one positive real zero. 
\end{lemma}

\begin{proof}
    Because $f(0) = -1$ and $\lim_{r\to\infty} f(r) = \infty$, by the Intermediate Value Theorem, $f$ has at least one positive real zero. Because $f'(r) = mr^{m-1}+ 2c k\alpha  r^{k-1} > 0$, by Rolle's Theorem, there can not be two positive real zeros. 
    \end{proof}

\begin{lemma}\label{dallinlemma}
Let $f$ be as in \eqref{realfunctions} with $k>0$, $j$ odd, and $\alpha <0$. Then for all $c>0$, $f$ has no positive real zeros. 
\end{lemma}
    \begin{proof}
    Because $\alpha < 0$, it is clear that for all positive $r$, $f(r) < -1$. Thus $f$ has no positive real zeros. 
    \end{proof}

\begin{lemma}\label{marylemma}
Let $f$ be as in \eqref{realfunctions} with $k>0$, $j$ even, and $\alpha <0$. Then for all $c>0$, $f$ has exactly one positive real zero. 
\end{lemma}

    \begin{proof}
    Because $f(0) = -1$ and $\lim_{r\to\infty} f(r) = \infty$, by the Intermediate Value Theorem, $f$ has at least one positive real zero.  Observe,
    \[
    f'(r) = m r^{m-1} + 2c k \alpha r^{k-1} = r^{k-1}( m r^{m-k} +2c k \alpha).
    \]
    For positive $r$, the first factor is always positive. Because $\alpha < 0$, the second factor changes sign from negative to positive at the positive real number $r_0=\left(-\frac{2ck\alpha}{m}\right)^\frac{1}{m-k}$. Thus $f$ starts off negative, decreases to its global minimum at $r_0$, then increases monotonically to infinity as $r$ tends to infinity. Therefore, $f$ has exactly one positive real zero.   
    \end{proof}

\begin{lemma}\label{elilemma}
Let $f$ be as in \eqref{realfunctions} with $k>0$, $j$ odd, and $\alpha >0$.  There exists $c_0 > 0$ such that, for $0 < c < c_0$, $f$ has no positive real zeros and for $c > c_0$, $f$ has exactly two positive real zeros.
\end{lemma}

    \begin{proof}
    By a calculation nearly identical to the above,
    \[
    f'(r) = r^{k-1}\left(-m r^{m-k}+2 c k \alpha  \right)
    \]
     with $k,\alpha>0$.  Thus, there is exactly one positive critical point at $r_0 =\left({\frac{2c k\alpha }{m}}\right)^{\frac{1}{m-k}}$ at which $f'$ changes from positive to negative. Therefore, $r_0$ is the location of both the only local extreme and the global maximum of $f$.  We claim there exists $c_0$ such that when $0 < c < c_0$, $f(r_0) < 0$ and when $c > c_0$, $f(r_0) > 0$. Observe,
    \begin{align*}
    f(r_0) &= -\left( \frac{2c k\alpha }{m} \right)   ^ {\frac{m}{m-k}}  + 2c\alpha \left( \frac{2c k\alpha }{m}   \right) ^ \frac{k}{m-k} - 1 \\
    &= \left(2c\alpha\right)^{\frac{m}{m-k}} \left( -\left(\frac{k}{m}\right) ^ {\frac{m}{m-k}} +  \left( \frac{k}{m} \right) ^ {\frac{k}{m-k}} \right) -1\\
    &:=c^{\frac{m}{m-k}}\beta - 1
    \end{align*}
    where 
    \[
    \beta = \left(2\alpha\right)^{\frac{m}{m-k}} \left( -\left(\frac{k}{m}\right) ^ {\frac{m}{m-k}} +  \left( \frac{k}{m} \right)^{\frac{k}{m-k}} \right).
    \]
    Because $\alpha >0$, $\left(2c\alpha\right)^{\frac{m}{m-k}} > 0$. Also, $\frac{k}{m} < 1$ and $\frac{m}{m-k} > \frac{k}{m-k}$. Hence $\beta >0$.  Thus, if $c < \beta^{-\frac{m-k}{m}}$, $f(r_0) < 0$ and if  $c > \beta^{-\frac{m-k}{m}}$, $f(r_0) > 0$. 
    
    Because $f(0)=-1$ and $\lim_{r \to \infty} f(r) = -\infty$, taking $c_0=\beta^{-\frac{m-k}{m}}$ it follows that for $0<c<c_0$, $f$ has no positive real zeros and for $c>c_0$, $f$ has exactly two positive real zeros. 
    \end{proof}

\subsection{Negative $k$}

\begin{lemma}\label{Lemma: kneg, j odd, alpha neg} 
  Let $f$ be as in \eqref{realfunctions} with $k<0$, $j$ odd, and $\alpha <0$. Then for all $c>0$, $f$ has no positive real zeros. 
\end{lemma}

  \begin{proof}
    Because $\alpha < 0 $, for all $r > 0$, $f(r) < -1$ and so $f$ has no positive real zeros.
    \end{proof}

\begin{lemma}\label{lemma: k neg, j even, alpha neg}
  Let $f$ be as in \eqref{realfunctions} with $k<0$, $j$ even, and $\alpha <0$. Then for all $c>0$, $f$ has exactly one positive real zero. 
\end{lemma}

    \begin{proof}
    Because $\lim_{r \to 0^+} f(r) = -\infty$ and $\lim_{r \to \infty} f(r) = \infty$, by the Intermediate Value Theorem, $f$ has at least one positive real zero. Furthermore, because $k<0$ and $\alpha <0$, $f'(r) = mr^{m-1}+2 c k \alpha r^{k-1}\ > 0$, and so by Rolle's Theorem, $f$ can not have two positive real zeros. 
    \end{proof}

\begin{lemma}\label{lemma: k neg, j odd, alpha pos}
  Let $f$ be as in \eqref{realfunctions} with $k<0$, $j$ odd, and $\alpha > 0$. Then for all $c>0$, $f$ has exactly one positive real zero. 
\end{lemma}

    \begin{proof}
    Because $\lim_{r \to 0^+} f(r) = \infty$ and $\lim_{r \to \infty} f(r) = -\infty$, by the Intermediate Value Theorem, $f$ has at least one positive real zero.  Furthermore, because $k<0$ and $\alpha>0$, $f'(r) = -mr^{m-1}+ 2 c k \alpha r^{k-1} < 0$, by Rolle's Theorem, $f$ can not have two positive real zeros.
    \end{proof}

\begin{lemma}\label{lemma: k neg, j even, alpha pos}
  Let $f$ be as in \eqref{realfunctions} with $k<0$, $j$ even, and $\alpha > 0$. There exists $c_0>0$ such that, for all $0<c <c_0$, $f$ has two positive real zeros and for $c>c_0$, $f$ has no positive real zeros. 
\end{lemma}

    \begin{proof}
   As above,
    \[ 
        f'(r) =r^{k-1}(m r^{m-k} + 2c k \alpha ).
    \]
    Because $k<0$ and $\alpha>0$, the only positive critical point is at $r_0 = \left(\frac{2c |k| \alpha }{m}\right)^\frac{1}{m-k}$. Also, for $r < r_0$, $f'(r) < 0$ and for $r > r_0$, $f'(r)>0$. Therefore, $r_0$ is the location of the global minimum and the only local extreme of $f$. We claim that there exists $c_0>0$ such that, if $0<c<c_0$, $f(r_0)<0$ and if $c>c_0$, $f(r_0)>0$.
    \begin{align*}
        f(r_0) &= (2c \alpha)^\frac{m}{m-k} \left(  \left( \frac{|k|}{m}\right)^\frac{m}{m-k}  +\left(\frac{|k|}{m} \right)^\frac{k}{m-k}    \right)-1\\
        &=c^\frac{m}{m-k} \beta - 1
    \end{align*}
    where
    \[
    \beta = (2 \alpha)^\frac{m}{m-k}\left(  \left( \frac{|k|}{m}\right)^\frac{m}{m-k}  +\left(\frac{|k|}{m} \right)^\frac{k}{m-k}    \right).
    \]
    Because $\alpha > 0$, $\beta >0$. Let $c_0=\beta^{-\frac{m-k}{m}}$. Then if $0<c<c_0$, $f(r_0)<0$ and if $c>c_0$, $f(r_0)>0$. Because $\lim_{r \to 0^+} f(r) = \infty$, $\lim_{r \to \infty} f(r)= \infty$, and $f(f_0)$ is the global minimum of $f$, when $0<c < c_0$, $f$ has precisely two positive real zeros and when $c>c_0$, $f$ has no positive real zeros. 
    \end{proof}

\section{Counting Cases}

Recall, by Lemma \ref{lemma1}, the zeros of our function $p$ lie on the $2m$ rays in the complex plane with angles $\theta = \frac{j\pi}{m}$ for $j = 0,1,2, \dots,  2m-1$. We showed in Sections 2 and 3 that, associated with each ray, is a real-valued function $f_j$. When $r$ is a positive real zero of $f_j$, the original function $p$ has a zero on the $j$-th ray a distance $r$ from the origin. Thus counting the zeros of $p$ is equivalent to counting the positive real zeros of the $f_j$. The difficulty comes from the fact that $f_j$ can take one of six forms depending on the parity of $j$, the sign of $k$, and whether $\alpha = \cos \left(\frac{k j \pi}{m} \right)$ is positive, negative, or zero. This last dependence is particularly difficult to unravel. More specifically, the function $t \mapsto \cos \left( \frac{k \pi}{m} t \right)$ completes $k$ cycles on an interval of length $2m$ and determining how many integer values of $t$ on this interval make it positive, negative, or zero depends on some number-theoretic properties of $k$ and $m$. We address this question in this section.

The key observation of this section is that 
\[
\cos\left( 
 \frac{kj \pi}{m}\right) = \Re \left(e^{i\frac{ k j \pi}{m}} \right) =\Re \left( e^{i \frac{k \pi}{m}} \right)^j
 \]
 and so we can consider instead how the powers of certain roots of unity are distributed around the unit circle.  We thus recall some basic facts about roots of unity. Let $\omega$ be a primitive $q$-th root of unity. Then $\omega$ generates a cyclic group of order $q$ under multiplication, commonly denoted by $\langle \omega \rangle$, and the set $G = \{\omega^\ell : 0 \leq \ell \leq q-1\}$ contains all $q$ members of the group.  Thus there is a one-to-one correspondence between the set of $\omega^\ell$ for $j \in \{ 0, 1, \ldots, q-1\}$ and the $q$ equally-spaced points on the unit circle associated with angles of the form $\frac{2j\pi}{q}$. However, which $\omega^\ell$ is associated with which point depends on which primitive $q$-th root of unity we take and is difficult to describe. Fortunately, we do not need such precise information; we just need to know how many $\ell$ of a certain parity are associate with points in the right half-plane, in the left half-plane, and on the imaginary axis. The propositions in this section answer these questions. 

\begin{proposition}\label{rightside}
    Let $\omega$ be a primitive $q$-th root of unity. Then there are $2\lfloor\frac{q-1}{4}\rfloor + 1$ elements of $\langle \omega \rangle$ with positive real part, i.e., there are $2\lfloor\frac{q-1}{4}\rfloor + 1$ values of $\ell \in \{0, 1, \ldots, q-1\}$ for which $\Re \omega^\ell >0$. 
\end{proposition}
\begin{proof}
Because the $q$-th roots of unity are associated with angles of the form $\frac{2\pi j}{q}$ for $j \in \mathbb{Z}$, we need only count the integers $j$ for which
\[
-\frac{\pi}{2} < \frac{2 \pi j}{q} < \frac{\pi}{2},
\]
or, equivalently,
\[
-\frac{q}{4} < j < \frac{q}{4}. 
\]
There are $2 \lfloor \frac{q-1}{4} \rfloor+1$ such integers. 
\end{proof}

%
%
%

For our next results, the parity of $k$ is important. We must understand powers of $\omega := e^{i\frac{k \pi}{m}} = e^{i\frac{2 k \pi}{2m}}$. Throughout this paper, we assume $\gcd(m,k) = 1$, and so if $k$ is odd, $\gcd(2m,k) =1$. In this case, $\omega$ is a primitive $2m$-th root of unity. If, however, $k$ is even, $\omega$ is a primitive $m$-th root of unity. 

\subsection{Odd $k$}

Throughout this section, $\omega = e^{i\frac{2k\pi}{2m}}$. Because $\gcd(2m,k) =1$, $\omega$ is a primitive $2m$-th root of unity and so $G=\{\omega^j : 0 \leq j \leq 2m-1\}$ includes all $2m$-th roots of unity. 

For the next proposition, we introduce some new terminology; we call $\omega^j$ and $\omega^{j'}$ {\em adjacent} $2m$-th roots of unity if they are adjacent on the unit circle, i.e., if their arguments differ by $\frac{2\pi}{2m}$ modulo $2\pi$. 

\begin{proposition}\label{neighbors}
Let $j,j' \in \{0,1,... 2m-1\}$. If $\omega^{j}$ and $\omega^{j'}$ are adjacent $2m$-th roots of unity, then $j$ and $j'$ have opposite parity.
\end{proposition}

    \begin{proof}
    Because $\omega^{j}$ and $\omega^{j'}$ are adjacent, there exists an integer $\ell$ such that
    \[
    \frac{2\pi{kj}}{2m} - \frac{2\pi{kj'}}{2m} = \frac{2\pi}{2m} + 2 \ell \pi
    \]
    Thus
    \[
    (j-j')k=1+2\ell m. 
    \]
    Because the right-hand side is odd, so is the left-hand side. Thus, $j-j'$ is odd, meaning that $j$ and $j'$ have opposite parity. 
    \end{proof}

\subsection{Even $k$}

Because $\gcd(m,k)=1$ and $k$ is even, $m$ is odd.  We write $k = 2l$ for some $l \in \mathbb{Z}$. 
Because $\gcd(m,l) = 1$, $\omega =e^{i\frac{k \pi}{m}} =  e^{i\frac{2 l \pi }{m}}$ is a primitive $m$-th root unity. Therefore, for $j \in \{0,1,\dots,  2m-1\}$, $\omega^j$ will hit every $m$-th root unity twice. We now prove the following proposition about the parity of $j,j' \in \{0,1,\dots,  2m-1\}$ such that $\omega^j = \omega^{j'}$.

\begin{proposition}\label{buddies}
For each $j \in \{0,1, \dots,  m-1\}$, there exists a $j' \in \{m,1, \dots,  2m-1\}$ of opposite parity such that $\omega^j = \omega^{j'}$.
\end{proposition}

\begin{proof} If $j \in \{0,1,\dots,  m-1\}$, then $j' = j + m \in \{m,m+1, \dots,  2m-1\}$, and because $m$ is odd, $j$ and $j'$ have opposite parity. Furthermore, as $\omega$ is an $m$-th root of unity, 
\[
\omega^{j'} = \omega^{j+m}=\omega^j. 
\]
\end{proof}
    
In the next subsection, we use \ref{rightside}, \ref{neighbors}, and \ref{buddies} to count even and odd $j$ such that $\cos({\frac{kj\pi}{m}})$ is positive, negative, or zero.








\subsection{Proof of Theorem \ref{theorem1}}

We are now equipped to prove Theorem \ref{theorem1}.

\begin{proof}
Recall, we are counting zeros of $p(z) = z^m+c(z^k+\overline{z}^k)-1$ for $m > k$, $c \in \mathbb{R_+}$, and $\gcd(m,k) = 1$.  By Lemma \ref{lemma1}, the zeros of $p$ lie on the $2m$ rays in the complex plane with angles $\frac{j\pi}{m}$ for $j = 0,1,2,...,2m-1$. We showed in Section 2 that each ray is associated with a real polynomial $f_j$ taking one of six forms depending on the parity of $j$ and whether $\alpha$ is positive, negative, or zero. In Section 3, we counted the positive real zeros for each of these six types of polynomial. It remains to determine the actual form of $f_j$ for each $j \in \{0,2,\ldots, 2m-1\}$. The question is whether $\Re \omega^j$ is positive, negative, or zero for $\omega = e^{i\frac{k \pi}{m}}$. 

A complete proof of the theorem would require many cases. We give the proof for one of the more complicated cases and leave the others to the reader.

Consider the case in which $m \equiv 1 \mod 4$ and $k$ is even. Thus, $k = 2l$ for some $l \in \mathbb{N}$ and $m=4d+1$ for some $d \in \mathbb{N}$. Thus $\omega = e^{i\frac{2l\pi}{m}}$ and $\omega$ is a primitive $m$-th root of unity. By Proposition \ref{rightside}, there are 
\[
2\left\lfloor\frac{m-1}{4}\right\rfloor + 1 = 2d+1 
\]
$m$-th roots of unity with a positive real part. Furthermore, by Proposition \ref{buddies}, for any $m$-th root of unity $\beta$, there exist $j,j' \in \{0, 1, \ldots, 2m-1\}$ of opposite parity with $\beta  = \omega^j = \omega^{j'}$. Therefore, there are $2d+1$ even $j$ and $2d+1$ odd $j$ in $\{0,1,\ldots,2m-1\}$ such that $\Re \omega^j > 0$. Consequently, there are $m-(2d+1)=2d$ even $j$ and $2d$ odd $j$ in $\{0,1,\ldots,2m-1\}$ such that $\Re \omega^j \leq 0$.   





By Lemmas 2--7, 
we have that
\begin{align*}
\text{Even } j, &\quad \Re \omega^j  = 0 \hspace{1em} \rightarrow \hspace{1em} f_j(r) = r^m - 1, \quad \alpha = 0 \hspace{1em} \rightarrow \hspace{1em} 1 \text{ zero } \forall c \\[1em]
\text{Odd } j, &\quad \Re \omega^j = 0 \hspace{1em} \rightarrow \hspace{1em} f_j(r) = -r^m - 1, \quad \alpha = 0 \hspace{1em} \rightarrow \hspace{1em} \text{No zeros } \forall c \\[1em]
\text{Even } j, &\quad \Re\omega^j < 0 \hspace{1em} \rightarrow \hspace{1em} f_j(r) = r^m + 2c\alpha r^k - 1, \quad \alpha < 0 \hspace{1em} \rightarrow \hspace{1em} 1 \text{ zero } \forall c \\[1em]
\text{Odd } j, &\quad \Re\omega^j < 0 \hspace{1em} \rightarrow \hspace{1em} f_j(r) = -r^m + 2c\alpha r^k - 1, \quad \alpha < 0 \hspace{1em} \rightarrow \hspace{1em} \text{No zeros } \forall c \\[1em]
\text{Even } j, &\quad \Re \omega^j > 0 \hspace{1em} \rightarrow \hspace{1em} f_j(r) = r^m + 2c\alpha r^k - 1, \quad \alpha > 0 \hspace{1em} \rightarrow \hspace{1em} 1 \text{ zero } \forall c \\[1em]
\text{Odd } j, &\quad \Re \omega^j > 0 \hspace{1em} \rightarrow \hspace{1em} f_j(r) = -r^m + 2c\alpha r^k - 1, \quad \alpha > 0 \hspace{1em} \rightarrow \hspace{1em} 0 \text{ to } 2 \text{ zeros }
\end{align*}


Therefore, the total number of zeros of $p$ can be calculated by multiplying the number of $j$ that give rise to each case and the number of zeros associated with that case. Thus, when $c$ is small, the number of zeros is
\[
1(2d) + 1(2d + 1) = 4d + 1 = m.
\]
As $c$ increases, eventually for each odd $j$ with $Re \omega^j > 0$ there are 2 new zeros of $p$. Therefore, when $c$ is sufficiently large, the maximum number of zeros of $p$ is
\[ 
m + 2(2d+1) = m + 4d + 2 = 2m + 1.
\]
Therefore, as $c$ increases, the number of zeros of $p$ increases monotonically from a minimum of $m$ to a maximum of $2m + 1$.
\end{proof}

The proofs of the other cases rely on a similar strategy: Use Proposition \ref{rightside} to count the roots of unity with positive real part. Then, using Propositions \ref{neighbors} and \ref{buddies}, count how many of each come from odd $j$ and how many come from even $j$. Then, using lemmas \ref{degposlemma} -- \ref{elilemma}, count the zeros associated with each case. By multiplying the number of $j$ that give rise to each case by the number of zeros associated with that case, the result follows.

\subsection{Proof of Theorem \ref{theorem2}}

We now use a similar technique to prove Theorem \ref{theorem2}.

\begin{proof}
Now we are counting the zeros of  $p(z) = z^m+c(z^{k}+\overline{z}^{k})-1$ {\em for $k$ negative}, $m > |k|$, $c \in \mathbb{R}_+$, and $\gcd(m,k) = 1$. By Lemma \ref{lemma1}, the zeros of $p$ lie on the $2m$ rays in the complex plane with angles $\frac{j\pi}{m}$ for $j = 0,1,2,...2m-1$. We showed in Section 2 that each ray is associated with a real function $f_j$ taking one of six forms depending on the parity of $j$ and whether $\alpha$ is positive, negative, or zero. In Section 3, we counted the positive real zeros for each of these six types of function. It remains to determine the actual form of $f_j$ for each $j \in \{0,2,\ldots, 2m-1\}$. The question is whether $\Re  \omega^j$ is positive, negative, or zero for $\omega = e^{i\frac{k \pi}{m}}$. 

As above, we give the proof for one of the more complicated cases and leave the others to the reader.

Consider the case in which $m \equiv 3 \mod 4$ and $k$ is odd. Thus $m = 4d + 3$ for some $d \in \mathbb{N}$ and $\omega = e^{i\frac{2 k \pi}{2m}}$ is a primitive $2m$-th root of unity. By Proposition \ref{rightside}, there are 
\[
2\left\lfloor\frac{2m-1}{4}\right\rfloor + 1 = 4d + 3 = m 
\]
$2m$-th roots of unity with positive real part. By Proposition \ref{neighbors}, the $j$ associated with adjacent roots of unity have opposite parity. Because $\omega^0 = 1$ and $0$ is even, as we move around the unit circle starting at 1, the $j$ associated with the roots of unity alternate in parity. Thus of the $4d+3$ roots of unity for which $\Re \omega^j >0$, $2d+1$ are associated with even $j$ and $2d+2$ are associated with odd $j$. Because there are $m$ even $j$ and $m$ odd $j$ in $\{0,1,\ldots,2m-1\}$, it follows that there are $2d+2$ even $j$ for which $\Re \omega^j < 0$ and $2d+1$ odd $j$ for which $\Re \omega^j < 0$. For $m=4d+3$, it is easy to check that neither $i$ nor $-i$ is a $2m$-th root of unity, and so there are no $j$ for which $\Re \omega^j = 0$.  



By Lemmas 8--11, 
we have that
\begin{align*}
\text{Even } j, &\quad \Re\omega^j < 0 \hspace{1em} \rightarrow \hspace{1em} f_j(r) = r^m + 2c\alpha r^k - 1, \quad \alpha < 0 \hspace{1em} \rightarrow \hspace{1em} 1 \text{ zero } \forall c \\[1em]
\text{Odd } j, &\quad \Re\omega^j < 0 \hspace{1em} \rightarrow \hspace{1em} f_j(r) = -r^m + 2c\alpha r^k - 1, \quad \alpha < 0 \hspace{1em} \rightarrow \hspace{1em} \text{No zeros } \forall c \\[1em]
\text{Even } j, &\quad \Re\omega^j > 0 \hspace{1em} \rightarrow \hspace{1em} f_j(r) = r^m + 2c\alpha r^k - 1, \quad \alpha > 0 \hspace{1em} \rightarrow \hspace{1em} 2 \text{ to } 0  \text{ zeros }\\[1em]
\text{Odd } j, &\quad \Re\omega^j > 0 \hspace{1em} \rightarrow \hspace{1em} f_j(r) = -r^m + 2c\alpha r^k - 1, \quad \alpha > 0 \hspace{1em} \rightarrow \hspace{1em} 1 \text{ zero } \forall c
\end{align*}


Therefore, the total number of zeros of $p$ can be calculated by multiplying the number of $j$ that give rise to each case and the number of zeros that are produced by each case. Thus, when $c$ is small, the number of zeros is
\[
1(2d+2) + 1(2d+2) + 2(2d+1) = 2m.
\]
As $c$ increases, eventually for each even $j$ with $\Re \omega^j > 0$, two zeros of $p$ disappear. Therefore, when $c$ is sufficiently large, the number of zeros of $p$ is
\[
1(2d+2) + 1(2d+2) = m + 1.
\]
Therefore, as $c$ increases, the number of zeros of $p$ decreases monotonically from a maximum of $2m$ to a minimum of $m+1$.
\end{proof}

\section{Conclusion and Future Directions}

In this paper, we counted zeros for harmonic functions of the form $p(z) = z^m +c(z^k + \overline{z}^k) -1$. Using real-variable techniques, we derived an elegant result that describes how the number of zeros change with the positive real parameter $c$. In the polynomial case (positive $k$), the number of zeros increases as $c$ increases. On the other hand, for negative $k$, the number of zeros decreases as $c$ increases. In the past, such detailed zero-counting theorems have tended to focus on families for which the critical curve separating the sense-preserving and sense-reversing regions is a circle (\cite{BBDHS}, \cite{BDHPWW}, \cite{lee}). Although the critical curve for our family is not a circle, we were able to obtain a similarly detailed theorem because of the additional structure of all our zeros lying on rays. 

Our work opens several directions for further research. One natural extension is to consider the more general family
\[
p(z) = z^m +az^k +b\overline{z}^k -1
\]
where $a$ and $b$ are positive real numbers but $a \neq b$. For such functions, the zeros no longer neatly fall on rays and so the real-variable arguments used in this paper will not apply. Furthermore, the critical curve is complicated and one would expect it to be difficult to obtain a zero-counting theorem that applies to all parts of the $ab$-parameter space. (See \cite{work} and \cite{sam} for examples of the kind of analysis that might be possible.) Still, it would be interesting to know how the maximum number of zeros compares with the family studied in this paper. 

\section{Declarations}
\noindent \textbf{Funding}: Not Applicable \\
\textbf{Data Availability}: Not Applicable \\
\textbf{Conflict of Interest}: The authors declare that they have no conflict of interest.

\bibliographystyle{plain}
\bibliography{Real_Methods}

\end{document}